\documentclass[11pt,english]{article}
\usepackage[T1]{fontenc}
\usepackage[latin9]{inputenc}
\usepackage{geometry}
\usepackage{amsmath}
\usepackage{amsthm}
\usepackage{amssymb}
\usepackage{setspace}
\usepackage{esint}
\makeatletter
\numberwithin{equation}{section}
\numberwithin{figure}{section}
  \theoremstyle{plain}
  \newtheorem{thm}{\protect\theoremname}[section]
  \theoremstyle{definition}
  \newtheorem{example}{\protect\examplename}[section]
 \theoremstyle{definition}
 \newtheorem*{defn*}{\protect\definitionname}
  \theoremstyle{plain}
  \newtheorem*{thm*}{\protect\theoremname}
  \theoremstyle{plain}
  \newtheorem{lem}{\protect\lemmaname}[section]

\makeatother

\usepackage{babel}
  \providecommand{\definitionname}{Definition}
  \providecommand{\examplename}{Example}
  \providecommand{\lemmaname}{Lemma}
  \providecommand{\theoremname}{Theorem}
\providecommand{\theoremname}{Theorem}

\begin{document}

\title{Closed-form formulas for the distribution of the jumps of doubly-stochastic
Poisson processes}

\author{Arturo Valdivia\thanks{Universitat de Barcelona, Gran Via de les Corts Catalanes, 585, E-08007
Barcelona, Spain. E-mail\texttt{: arturo@valdivia.xyz}}}
\maketitle
\begin{abstract}
We study the obtainment of closed-form formulas for the distribution
of the jumps of a doubly-stochastic Poisson process. The problem is
approached in two ways. On the one hand, we translate the problem
to the computation of multiple derivatives of the Hazard process cumulant
generating function; this leads to a closed-form formula written in
terms of Bell polynomials. On the other hand, for Hazard processes
driven by L\'evy processes, we use Malliavin calculus in order to
express the aforementioned distributions in an appealing recursive
manner. We outline the potential application of these results in credit
risk.\\
\textbf{}
\end{abstract}

\section{Introduction}

Consider an ordered series of random times $\tau_{1}\leq...\leq\tau_{m}$
accounting for the sequenced occurence of certain events. In the context
of credit risk, these random times can be seen as \emph{credit events}
such as the firm's value sudden deterioration, credit rate downgrade,
the firm's default, etcetera. The valuation of defaultable claims
(see \cite{BieleckiRutkowski,SchCarCR}) is closely related to computation
of the quantities
\[
\mathbb{P}(\tau_{n}>T|\mathcal{F}_{t}),\qquad t\geq0,\quad n=1,...,m,
\]
where the \emph{reference filtration} $\mathbb{F}=(\mathcal{F}_{t})_{t\geq0}$
accounts for the information generated by all state variables.

An interesting possibility to model these random times consists in
considering $\tau_{1},...,\tau_{m}$ as the succesive jumps of a \emph{doubly-stochastic
Poisson process} (\emph{DSPP}). That is, a time-changed Poisson process
$(P_{\Lambda_{t}})_{t\ge0}$, where the time change $(\Lambda_{t})_{t\geq0}$
is a non-decreasing c\`adl\`ag $\mathbb{F}$-adapted process starting
at zero; and the Poisson process $(P_{t})_{t\ge0}$ has intensity
rate equal to $1$, and it is independent of $\mathbb{F}$. We refer
to $(\Lambda_{t})_{t\geq0}$ as the \emph{Hazard process}.

The purpose of this note is to study the obtainment of closed-form
formulas for the distributions of the $n$-th jump of a doubly-stochastic
Poisson process. We address the problem from two different approaches.
First, we relate this problem to the computation of the first $n$
derivatives of the Hazard process cumulant generating function. As
shown below, the result is written in closed-form in terms of \emph{Bell
polynomials} on the aforementioned derivatives ---see \cite{Comtet,Johnson,Riordan}
for details on these polynomials.
\begin{thm}
\label{thm: characteristic function}For $0\leq t<T$, denote the
cumulant generating function of $\Lambda_{T}$ by
\[
\Psi(u):=\log\mathbb{E}\left[\left.\exp\{\mathrm{\mathrm{i}}u\Lambda_{T}\}\right|\mathcal{F}_{t}\right].
\]
If $\Lambda_{T}$ has a finite conditional $n$-th moment (\emph{i.e.},
$\mathbb{E}\left[\left.\Lambda_{T}^{n}\right|\mathcal{F}_{t}\right]<\infty$),
then the following equation holds true
\begin{equation}
\mathbb{P}(\tau_{n}>T|\mathcal{F}_{t})=\mathbf{1}_{\{\tau_{n}>t\}}\sum_{k=0}^{n-1}\frac{\mathrm{e}^{\Psi(\mathrm{i})}}{k!\mathrm{i}^{k}}\mathbf{B}_{k}\left(\frac{\partial\Psi}{\partial u}(\mathrm{i}),...,\frac{\partial^{k}\Psi}{\partial u^{k}}(\mathrm{i})\right),\label{eq: n-th cox jump distribution}
\end{equation}
where $\mathbf{B}_{k}$ is the $k$-th Bell polynomial.
\end{thm}
In light of this result, two considerations are in order. On the one
hand, it is desirable to consider a model for $(\Lambda_{t})_{t\geq0}$
having a cumulant generating function $\Psi$ being analytic (around
$\mathrm{i}$), so that arbitrary jumps of the doubly-stochastic Poisson
process can be handled. On the other hand, it is straightforward to
compute (\ref{eq: n-th cox jump distribution}) in closed-form given
a tratactable expression for the cumulant generating function $\Psi$.
See examples in Section 2.

As a second approach, we compute the aforementioned distributions
direcly, by means of the Malliavian calculus. For this approach we
consider a strictly positive pure-jump L\'evy process $(L_{t})_{t\geq0}$
with L\'evy measure $\nu$, and having moments of all orders ---see
\cite{Applebaum,Sato} and \cite{dNOkPro} for a general exposition
about L\'evy processes and Malliavin calculus. We then assume that
the Hazard process is of the form 

\begin{equation}
\Lambda_{t}=\int_{0}^{t}\int_{\mathbb{R}_{0}}\sigma(s,z)\widetilde{N}(\mathrm{d}s,\mathrm{d}z),\qquad t\geq0,\label{eq: Levy driven Hazard process-1}
\end{equation}

where $\widetilde{N}$ is the compensated Poisson random measure associated
$(L_{t})_{t\geq0}$, and $\sigma$ is a deterministic function, integrable
with respect to $\widetilde{N}$. Assume further that $\mathbb{F}$
is given by the natural filtration generated by the driving L\'evy
process $(L_{t})_{t\geq0}$. In this setting, we have the following
result.
\begin{thm}
\label{thm: the recursive formula}The conditional distribution of
the $n$-th jump of doubly-stochastic Poisson process with Hazard
process satisfying (\ref{eq: Levy driven Hazard process-1}) is given
by
\[
\mathbb{P}(\tau_{n}>T|\mathcal{F}_{t})=\mathbf{1}_{\{\tau_{n}>t\}}\mathrm{e}^{\Lambda_{t}}\left(\sum_{k=0}^{n-1}\sum_{j=0}^{k}\frac{\Lambda_{t}^{j}}{j!(k-j)!}m_{k-j}(t)\right),
\]
where the quantities $m_{0},m_{1},...,m_{n}$ are given recursively
according to 
\begin{equation}
m_{0}(t):=\exp\left\{ \int_{t}^{T}\int_{\mathbb{R}_{0}}\left[\mathrm{e}^{-\sigma(s,z)}-1+\sigma(s,z)\right]\mathrm{d}s\nu(\mathrm{d}z)\right\} ,\label{eq:malliavin recursion start}
\end{equation}
and for $r\geq1$ 
\begin{align}
m_{r+1}(t) & =m_{r}(t)\int_{t}^{T}\int_{\mathbb{R}_{0}}(\mathrm{e}^{-\sigma(s,z)}-1)\sigma(s,z)\mathrm{d}s\nu(\mathrm{d}z)\label{eq:malliavin recursion}\\
 & \quad+\sum_{k=1}^{r}\binom{r}{k}m_{r-k}(t)\int_{t}^{T}\int_{\mathbb{R}_{0}}\mathrm{e}^{-\sigma(s,z)}\sigma^{k+1}(s,z)\mathrm{d}s\nu(\mathrm{d}z).\nonumber 
\end{align}

\end{thm}

The rest of the paper is organized as follows. In Section \ref{sec: Examples}
we present relevant examples appearing the literature. Finally in
Section 3 we provide the proofs of our results.

Let us remark that eventhough our study is motivated by the valuation
of defaultable claims, our results can potentially be also used in
other areas; see for instance \cite{Barraza,Lando,Zacks} and references
therein.

\section{Examples \label{sec: Examples}}

In many traditional models (\emph{e.g.}, \cite{DuffieLando}) the
Hazard processes $(\Lambda_{t})_{t\geq0}$ is assumed to be absolutely
continuous with respect to the Lebesgue measure, that is,
\begin{equation}
\Lambda_{t}:=\mathrm{\int_{0}^{t}\lambda_{s}\mathrm{d}s},\qquad t\geq0,\label{eq:abscont-1}
\end{equation}
where the process $(\lambda_{t})_{t\geq0}$ is usually refer to as
the \emph{hazard rate}, and it is seen as the instantaneous rate of
default in the credit risk context.\emph{ }The following two examples
show how to use Theorem using two prominent particular cases for the
hazard rate ---and consequently for the Hazard process.
\begin{example}
The\emph{ integrated square-root process} $(\Lambda_{t}^{intSR})_{t\geq0}$
(see \cite{Dufresne}) defined by means of (\ref{eq:abscont-1}) where
the hazard rate is given by the solution of
\[
\mathrm{d}\lambda_{t}^{SR}=\vartheta(\kappa-\lambda_{t}^{SR})\mathrm{d}t+\sigma\sqrt{\lambda_{t}^{SR}}\mathrm{d}W_{t},
\]
where $(W_{t})_{t\geq0}$ is a Brownian motion, and we assume $\sigma>0$
and $\vartheta\kappa\geq\sigma^{2}$ in order to ensure that $(\lambda_{t}^{SR})_{t\geq0}$
remains positive. Take now $\mathbb{F}$ as the natural filtration
generated by $(W_{t})_{t\geq0}$. It is well-known that the correspondent
Hazard process has an analytic cumulant generating function given
by 
\[
\Psi^{intSR}(u):=A(u,T-t)+\lambda_{t}^{SR}B(u,T-t),\qquad T\geq t\geq0,
\]
where the functions $A$ and $B$ are given by 
\[
A(u,T-t)=\frac{2\vartheta\kappa}{\sigma^{2}}\log\left(\frac{2\gamma\mathrm{e}^{\frac{1}{2}(\gamma+\vartheta)(T-t)}}{(\gamma+\vartheta)\mathrm{e}^{-\gamma(T-t)}-2\gamma}\right),\quad\text{and}\quad B(u,T-t)=\frac{2\gamma(\mathrm{e}^{-\gamma(T-t)}-1)}{(\gamma+\vartheta)\mathrm{e}^{-\gamma(T-t)}-2\gamma}
\]
with $\gamma:=\gamma(u):=\sqrt{\vartheta^{2}-2\mathrm{i}u\sigma^{2}}$.
The simplicity of $\Psi^{intSR}$ allows to compute its partial derivatives
involved in (\ref{eq: n-th cox jump distribution}). And finally we
can use the $n$-th Bell polynomial \emph{$\mathbf{B}_{n}$} characterization
given b\emph{y
\[
\mathbf{B}_{n}(x_{1},...,x_{n}):=\det\begin{bmatrix}\binom{n-1}{0}x_{1} & \binom{n-1}{1}x_{2} & \binom{n-1}{2}x_{3} & \cdots & \binom{n-1}{n-2}x_{n-1} & \binom{n-1}{n-2}x_{n}\\
-1 & \binom{n-2}{1}x_{1} & \binom{n-2}{1}x_{2} & \cdots & \binom{n-2}{n-3}x_{n-2} & \binom{n-2}{n-2}x_{n-1}\\
0 & -1 & \binom{n-3}{1}x_{1} & \cdots & \binom{n-3}{n-4}x_{n-3} & \binom{n-3}{n-3}x_{n-2}\\
\vdots & \vdots & \vdots &  & \vdots & \vdots\\
0 & 0 & 0 & \cdots & \binom{1}{0}x_{1} & \binom{1}{1}x_{2}\\
0 & 0 & 0 & \cdots & -1 & \binom{0}{0}x_{1}
\end{bmatrix},
\]
}where in each column the remaining entries below the $-1$ are equal
to zero. For instance, one can easily see that the first three Bell
polynomials are $\mathbf{B}_{1}(x_{1})=x_{1}$, $\mathbf{B}_{2}(x_{1},x_{2})=x_{1}^{2}+x_{2}$
and $\mathbf{B}_{3}(x_{1},x_{2},x_{3})=x_{1}^{3}+3x_{1}x_{2}+x_{3}$.
\end{example}

\begin{example}
\emph{The integrated non-Gaussian Ornstein-Uhlenbeck processes} (see
\cite{intOU}) defined by means of 
\begin{equation}
\Lambda_{t}^{intOU}:=\frac{1}{\vartheta}(1-\mathrm{e}^{-\vartheta t})\lambda_{0}+\frac{1}{\vartheta}\int_{0}^{t}(1-\mathrm{e}^{-\vartheta(t-s)})\mathrm{d}L_{\vartheta s},\qquad t\geq0,\label{eq:intOU}
\end{equation}
where $\vartheta,\lambda_{0}>0$ are free parameters, $\lambda_{0}$
being random, and $(L_{t})_{t\geq0}$ a non-decreasing pure-jump positive
L\'evy process. Equivalently, we can consider again the model in
(\ref{eq:abscont-1}) where this time $(\lambda_{t})_{t\geq0}$ is
given by the solution of
\[
\mathrm{d}\lambda_{t}=-\vartheta\lambda_{t}\mathrm{d}t+\mathrm{d}L_{\vartheta t},\qquad\lambda_{0}>0.
\]
An interesting property of this Hazard rate process is that it has
continuous sample paths. It can be shown that 
\begin{equation}
\Psi^{intOU}(u):=\frac{\mathrm{i}u\lambda_{0}}{\vartheta}(1-\mathrm{e}^{-\vartheta T})+\vartheta\int_{0}^{T}k_{L}\left(\frac{u}{\vartheta}(1-\mathrm{e}^{-\vartheta(T-s)})\right)\mathrm{d}s,\qquad T\geq0,\label{eq:log phi intOU}
\end{equation}
where we take $\mathcal{F}_{0}=\sigma(\lambda_{0})$, that is, the
$\sigma$-algebra generated by $\lambda_{0}$. Particular cases of
interest are the following. On the one hand, we have the so-called
\emph{Gamma$(a,b)$-OU process }which is obtained by taking $(L_{t})_{t\geq0}$
as a Compound Poisson process
\[
L_{t}=\sum_{n=1}^{Z_{t}}x_{n},\qquad t\geq0,
\]
where $(Z_{t})_{t\geq0}$ is a Poisson process with intensity $a\vartheta$,
and $(x_{n})_{n\geq1}$ is a sequence of independent identically distributed
$\mathrm{Exp}(b)$ variables. In this case, the correspondent Hazard
process in (\ref{eq:intOU}) has a finite number of jumps in every
compact time interval. Moreover, the equation (\ref{eq:log phi intOU})
becomes
\[
\Psi_{Gamma}^{intOU}(u):=\frac{\mathrm{i}u\lambda_{0}}{\vartheta}(1-\mathrm{e}^{-\vartheta T})+\frac{\vartheta a}{\mathrm{i}u-\vartheta b}\left(\left(b\log\left(\frac{b}{b-\frac{\mathrm{i}u}{\vartheta}(1-\mathrm{e}^{-\vartheta T})}\right)-\mathrm{i}uT\right)\right).
\]

On the other hand, we have the so-called \emph{Inverse-Gausssian}$(a,b)$\emph{-OU}
\emph{process} (see \cite{Nicolato} and Tompkins and Hubalek (2000))
which is obtained by taking $(L_{t})_{t\geq0}$ as the sum of two
independent processes, $(L_{t}=L_{t}^{(1)}+L_{2}^{(2)})_{t\geq0}$,
where $(L_{t}^{(1)})_{t\geq0}$ is an Inverse-Gausssian$(\frac{1}{2}a,b)$
process, and $(L_{t}^{(2)})_{t\geq0}$ is a Compound Poisson process
\[
L_{t}^{(2)}=b^{-1}\sum_{n=1}^{Z_{t}}x_{n}^{2},\qquad t\geq0,
\]
where $(Z_{t})_{t\geq0}$ is a Poisson process with intensity $\frac{1}{2}ab$,
and $(x_{n})_{n\geq1}$ is a sequence of independent identically distributed
$\mathrm{Normal}(0,1)$ variables. In this case, the correspondent
Hazard process in (\ref{eq:intOU}) jumps infinitely often in every
interval. Moreover, the equation (\ref{eq:log phi intOU}) becomes
\[
\Psi_{IG}^{intOU}(u):=\frac{\mathrm{i}u\lambda_{0}}{\vartheta}(1-\mathrm{e}^{-\vartheta T})+\frac{2a\mathrm{i}u}{b\vartheta}A(u,T),
\]

where, using $c:=-2b^{-2}\mathrm{i}u\vartheta^{-1}$, the function
$A$ is defined by
\begin{align*}
A(u,T):= & \frac{1-\sqrt{1+c(1-\mathrm{e}^{-\vartheta T})}}{c}\\
 & +\frac{1}{\sqrt{1+c}}\left[\mathrm{arctanh}\left(\frac{1-\sqrt{1+c(1-\mathrm{e}^{-\vartheta T})}}{c}\right)-\mathrm{arctanh}\left(\frac{1}{\sqrt{1+c}}\right)\right].
\end{align*}

In both of the cases above, we can see that the simplicity of $\Psi$
allows to compute (\ref{eq: n-th cox jump distribution}) in a straightforward
way.
\end{example}
This traditional approach reduces the analytical tractability of the
model, along with its parameters calibration. Indeed, suffices to
say the Laplace transform of a Hazard process as in (\ref{eq:abscont-1})
is known in closed-form only for a reduced number of Hazard rates
models. That is one the reasons why in more recent contributions the
modelling focus is set on the Hazard process itself, without requiring
to make a reference to the Hazard rate ---see for intance \cite{Bianchi,Nicolato}.
In this line, consider a Hazard process $(\Lambda_{t})_{t\geq0}$
as given in (\ref{eq: Levy driven Hazard process-1}). The following
example provides an explicit computation the quantities involved in
Theorem \ref{thm: the recursive formula}.
\begin{example}
(\emph{CMY Hazard process}) In the financial literature, the \emph{CMY
process} \emph{---or one-sided CGMY process} \cite{CGMY}--- with
parameters $C,M>0$ and $Y<1$ refers to the positive pure-jump L\'evy
process $(L_{t}^{CMY})_{t\geq0}$ having L\'evy measure $\nu_{CMY}$
given by
\[
\nu_{CMY}(z):=\frac{C\mathrm{e}^{-Mz}}{z^{1+Y}}\mathbf{1}_{\{z>0\}}.
\]
The \emph{Gamma process} and the \emph{Inverse Gaussian process} can
be seen as particular cases by taking $Y=0$ and $Y=\frac{1}{2}$,
respectively, see \cite{SchCarCR}.

Consider now a Hazard process of the form
\[
\Lambda_{t}^{CMY}:=\int_{0}^{t}\sigma(s)\mathrm{d}L_{s}^{CMY},\qquad t\geq0.
\]
This is equivalent to take, in (\ref{eq: Levy driven Hazard process-1}),
a function $\sigma$ is of the form $\sigma(s,z)=z\sigma(s)$. Then
the quantities in (\ref{eq:malliavin recursion start}) and (\ref{eq:malliavin recursion})
are given by
\[
m_{0}^{CMY}(t)=\begin{cases}
\exp\left\{ C\int_{t}^{T}\Gamma(1-Y)M^{Y-1}\sigma(s)+\Gamma(-Y)\left[(M+\sigma(s))^{Y}-M^{Y}\right]\mathrm{d}s\right\} , & Y\neq0\\
\exp\left\{ C\int_{t}^{T}\frac{\sigma(s)}{M}-\log\left(1+\frac{\sigma(s)}{M}\right)\mathrm{d}s\right\} , & Y=0
\end{cases}
\]
and
\begin{align*}
m_{n+1}^{CMY}(t) & =m_{n}^{CMY}(t)C\Gamma(1-Y)\left[\int_{t}^{T}\sigma(s)\left((M+\sigma(s))^{Y-1}-M{}^{Y-1}\right)\mathrm{d}s\right]\\
 & \quad+\sum_{k=1}^{n}\binom{n}{k}m_{n-k}^{CMY}(t)C\Gamma(k+1-Y)\int_{t}^{T}\sigma^{k+1}(s)(M+\sigma(s))^{Y-(k+1)}\mathrm{d}s.
\end{align*}
for $n\geq1$.

\emph{}
\end{example}

Finally, let us remark that we when considering a model like (\ref{eq: Levy driven Hazard process-1}),
the quantities appearing in Theorem \ref{thm: characteristic function}
and Theorem \ref{thm: the recursive formula} can be related according
to the following. 
\begin{example}
\label{cor: char func} Let the Hazard process $(\Lambda_{t})_{t\geq0}$
be given as in (\ref{eq: Levy driven Hazard process-1}). It can be
seen that in this case (Lemma \ref{lem: The-conditional-characteristic}
below) the cumulant generating function is given by
\[
\Psi(u)=\mathrm{\mathrm{i}}u\Lambda_{t}+\int_{t}^{T}\int_{\mathbb{R}_{0}}\left[\mathrm{e}^{\mathrm{\mathrm{i}}u\sigma(s,z)}-1-\mathrm{\mathrm{i}}u\sigma(s,z)\right]\mathrm{d}s\nu(\mathrm{d}z).
\]
Consequently, if the function $\sigma$ has finite moments 
\begin{equation}
\int_{0}^{T}\int_{\mathbb{R}_{0}}\sigma^{k}(s,z)\mathrm{d}s\nu(\mathrm{d}z)<\infty,\qquad k=1,...,n,\label{eq: assumption on sigma}
\end{equation}
then the $n$-th derivative of $\Psi$ is given by 
\[
\frac{1}{\mathrm{\mathrm{i}}}\frac{\partial\Psi}{\partial u}=\Lambda_{t}+\int_{t}^{T}\int_{\mathbb{R}_{0}}\left[\mathrm{e}^{\mathrm{\mathrm{i}}u\sigma(s,z)}-1\right]\sigma(s,z)\mathrm{d}s\nu(\mathrm{d}z),
\]
and
\[
\frac{1}{\mathrm{\mathrm{i}}^{k}}\frac{\partial^{k}\Psi}{\partial u^{k}}=\int_{t}^{T}\int_{\mathbb{R}_{0}}\mathrm{e}^{\mathrm{\mathrm{i}}u\sigma(s,z)}\sigma^{k}(s,z)\mathrm{d}s\nu(\mathrm{d}z),\qquad k=2,...,n.
\]
Indeed, these equations can be obtained by succesive differentiation
under the integral sign due to the assumption (\ref{eq: assumption on sigma}).
\end{example}

\section{Proofs\label{sec: Proofs}}

Let us start by the construction of the doubly-stochastic Poisson
process that we shall consider in what follows.

Let $\mathbb{F}=(\mathcal{F}_{t})_{t\geq0}$\emph{ }denote our reference
filtration; we shall assume that it satisfies the usual conditions
of $\mathbb{P}$-completeness and right-continuity. Let the i.i.d.
random variables $\eta_{1},...,\eta_{m}$ be exponentially distributed
with parameter $1$, all being independent of $\mathcal{F}_{\infty}$.
Then the $n$-th jump of the doubly-stochastic Poisson process with
Hazard rate $(\Lambda_{t})_{t\geq0}$ can be characterized as
\begin{equation}
\tau_{n}=\inf\left\{ t>0\::\;\Lambda_{t}\geq\eta_{1}+...+\eta_{n}\right\} .\label{eq: n-th jump-1}
\end{equation}

This construction leads to the following expression for the conditional
distribution of the DSPP $n$-th jump
\begin{equation}
\mathbb{P}(\tau_{n}>T|\;\mathcal{F}_{T})=\mathrm{e}^{-\Lambda_{T}}\sum_{j=0}^{n-1}\frac{1}{j!}\Lambda_{T}^{j},\qquad T\geq0.\label{eq: n-th jump survival probability-1}
\end{equation}

Indeed, by construction,
\[
\mathbb{P}(\tau_{n}>t|\;\mathcal{F}_{\infty})=\mathbb{P}\Bigg(\sum_{j=1}^{n}\eta_{j}>\Lambda_{t}\Bigg|\;\mathcal{F}_{\infty}\Bigg)=\mathrm{e}^{-\Lambda_{t}}\sum_{j=0}^{n-1}\frac{\Lambda_{t}^{j}}{j!},
\]
since conditioned to $\mathcal{F}_{\infty}$ the random variable $\eta_{1}+...+\eta_{n}$
has an Gamma distribution. The result then follows by preconditioning
to $\mathcal{F}_{t}$ ---recall that $(\Lambda_{t})_{t\geq0}$ is
$\mathbb{F}$-adapted.

Notice first that by conditioning (\ref{eq: n-th jump survival probability-1})
to $\mathcal{F}_{t}$ we get
\begin{equation}
\mathbb{P}(\tau_{n}>T|\;\mathcal{F}_{t})=\sum_{j=0}^{n-1}\frac{1}{j!}\mathbb{E}\left[\left.\Lambda_{T}^{j}\mathrm{e}^{-\Lambda_{T}}\right|\mathcal{F}_{t}\right],\qquad T\geq t\geq0.\label{eq: n-th jump survival probability}
\end{equation}
Then the purpose of Theorem \ref{thm: characteristic function} and
Theorem \ref{thm: the recursive formula} is to provide a way to compute
the conditional expectations in the equation above.

\subsection{Proof of Theorem \ref{thm: characteristic function}}

Let $\mu_{t}$ stand for the conditional (to $\mathcal{F}_{t}$) law
of $\Lambda_{T}$, so that the assumption on the $n$-th conditional
moment reads
\[
\int_{\mathbb{R}}x^{n}\mu_{t}(\mathrm{d}x)<\infty.
\]
As in the unconditional case (\emph{cf.} \cite[Theorem 13.2]{JacPro}),
the condition above ensures that the conditional characteristic function
\[
\varphi(u;t,T):=\mathbb{E}\left[\left.\exp\{\mathrm{\mathrm{i}}u\Lambda_{T}\}\right|\mathcal{F}_{t}\right]
\]
has continuous partial derivatives up to order $n$, and furthermore
the following equation holds true
\[
\frac{1}{\mathrm{i}^{k}}\frac{\partial^{k}\varphi(u;t,T)}{\partial u^{k}}=\mathbb{E}\left[\left.\Lambda_{T}^{k}\mathrm{e}^{\mathrm{i}u\Lambda_{T}}\right|\mathcal{F}_{t}\right],\qquad k=0,1,...,n.
\]
Now, let us recall that the\emph{ $n$-th Bell polynomial} \textbf{$\mathbf{B}_{n}$
}can also be written as
\begin{equation}
\mathbf{B}_{n}(x_{1},...,x_{n})=\sum_{k=0}^{n}\mathbf{B}_{n,k}(x_{1},...,x_{n-k+1}),\label{eq: Bell polynomials}
\end{equation}
where $\mathbf{B}_{n,k}$ stands for th \emph{partial} \emph{$(n,k)$-th
Bell polynomial}, \emph{i.e.}, $\mathbf{B}_{0,0}:=1$ and
\[
\mathbf{B}_{n,k}(x_{1},...,x_{n-k+1}):=\sum\frac{n!}{j_{1}!j_{2}!\cdots j_{n-k+1}!}\left(\frac{x_{1}}{1!}\right)^{j_{1}}\left(\frac{x_{2}}{2!}\right)^{j_{1}}\cdots\left(\frac{x_{n-k+1}}{(n-k+1)!}\right)^{j_{n-k+1}},
\]
where the sum runs over all sequences of non-negative indices such
that $j_{1}+j_{2}+\cdots=k$ and $j_{1}+2j_{2}+3j_{3}+\cdots=n$.
Using the Bell polynomials we have an expression for the chain rule
for higher derivatives:
\[
\frac{\mathrm{d}^{n}}{\mathrm{d}x^{n}}f\circ g=\sum_{k=0}^{n}(f^{(k)}\circ g)\mathbf{B}_{n,k}(g^{(1)},...,g^{(n+k-1)}),
\]
where the superscript denotes the correspondent derivative, \emph{i.e.},
$f^{(k)}:=\frac{\mathrm{d}^{k}}{\mathrm{d}x^{k}}f$ and $g^{(k)}:=\frac{\mathrm{d}^{k}}{\mathrm{d}x^{k}}g$,
which are assumed to exist. This expression is known as the \emph{Riordan's
formula} ---for these results on Bell polynomials we refer to \cite{Comtet,Johnson,Riordan}.

It remains to apply Riordan's formula to $f=\mathrm{exp}$ and $g=\Psi$
in order to get
\[
\frac{\mathrm{d}^{n}}{\mathrm{d}x^{n}}\varphi=\sum_{k=0}^{n}\varphi\mathbf{B}_{n,k}(\Psi^{(1)},...,\Psi^{(n+k-1)})=\varphi\mathbf{B}_{n}(\Psi^{(1)},...,\Psi^{(n)}),
\]
where the last equivalence follows from (\ref{eq: Bell polynomials}).

\subsection{Proof of Theorem \ref{thm: the recursive formula}}

From this moment on, we shall work with a strictly positive pure-jump
L\'evy process $(L_{t})_{t\geq0}$ having a L\'evy measure $\nu$
satisfying
\[
\int_{(-\varepsilon,\varepsilon)}\mathrm{e}^{pz}\nu(\mathrm{d}z)<\infty
\]
for every $\varepsilon>0$ and certain $p>0$. This condition implies
in particular that $(L_{t})_{t\geq0}$ have moments of all orders,
and the polynomials are dense in $L^{2}(\mathrm{d}t\times\nu)$. Notice
that this condition is always satisfied if the L\'evy measure has
compact support.

In other to prove the corollary we need the following.

\subsubsection{Preliminaries on Malliavin calculus via chaos expansions}

Let us now introduce basic notions of Malliavin calculus for Lévy
processes which we shall use as a framework. Here we mainly follow
\cite{dNOkPro}.

For every $T>0$, let $\mathcal{L}_{T}^{2}((\mathrm{d}t\times\nu)^{n}):=\mathcal{L}^{2}(([0,T]\times\mathbb{R}_{0})^{n})$
be the space of determinisitic functions such that
\[
\left\Vert f\right\Vert _{\mathcal{L}_{T}^{2}((\mathrm{d}t\times\nu)^{n})}:=\left(\int_{([0,T]\times\mathbb{R}_{0})^{n}}f^{2}(t_{1},z_{1},...,t_{n},z_{n})\mathrm{d}t_{1}\nu(\mathrm{d}z_{1})\cdots\mathrm{d}t_{n}\nu(\mathrm{d}z_{n})\right)^{\frac{1}{2}}<\infty,
\]
and such that they are zero over \emph{$k$-diagonal sets}, see \cite[Remark 2.1]{SoleUtzetVives}.
The \emph{symmetrization} $\tilde{f}$ of $f$ is defined by
\[
\tilde{f}(t_{1},z_{1},...,t_{n},z_{n}):=\frac{1}{n!}\sum_{\sigma}f(t_{\sigma(1)},z_{\sigma(1)},...,t_{\sigma(n)},z_{\sigma(n)}),
\]
where the sum runs over all the permutations $\sigma$ of $\{1,...,n\}$.
For every $f$ in the subspace of symmetric functions, $\widetilde{\mathcal{L}}_{T}^{2}((\mathrm{d}t\times\nu)^{n}):=\{f\in\mathcal{L}^{2}((\mathrm{d}t\times\nu)^{n})\;:\;f=\tilde{f}\}$,
we define the $n$-\emph{fold iterated integral of $f$ by}
\[
I_{n}(f):=n!\int_{0}^{T}\int_{\mathbb{R}_{0}}\cdots\int_{0}^{t_{2}}\int_{\mathbb{R}_{0}}f(t_{1},z_{1},...,t_{n},z_{n})\widetilde{N}(\mathrm{d}t_{1},\mathrm{d}z_{1})\cdots\widetilde{N}(\mathrm{d}t_{n},\mathrm{d}z_{n}).
\]
For constant values $f_{0}\in\mathbb{R}$ we set $I_{0}(f_{0}):=f_{0}$.
In these terms, the \emph{Wiener-Itô chaos expansion for Poisson random
measures,} due to \cite{Ito}, states that every $\mathcal{F}_{T}$-measurable
random variable $F\in\mathcal{L}^{2}(\mathbb{P})$ admits a representation
\[
F=\sum_{n=0}^{\infty}I_{n}(f_{n})
\]
via a unique sequence of elements $f_{n}\in\widetilde{\mathcal{L}}_{T}^{2}((\mathrm{d}t\times\nu)^{n})$.
In virtue of this result, each random field $(X_{t,z})_{(t,z)\in[0,T]\times\mathbb{R}_{0}}$
has an expassion
\[
X_{t,z}=\sum_{n=0}^{\infty}I_{n}(f_{n}(\cdot,t,z)),\qquad f_{n}(\cdot,t,z)\in\widetilde{\mathcal{L}}_{T}^{2}((\mathrm{d}t\times\nu)^{n}),
\]
provided, of course, that $X_{t,z}$ is $\mathcal{F}_{T}$-measureble
with $\mathbb{E}[X_{t,z}^{2}]<\infty$ for all $(t,z)$ in $[0,T]\times\mathbb{R}_{0}$.
Now we are in position to define the \emph{Skorohod integral} and
the \emph{Malliavin derivative.}
\begin{defn*}
The random field $(X_{t,z})_{(t,z)\in[0,T]\times\mathbb{R}_{0}}$
\emph{belongs to} $Dom(\delta)$ if
\[
\sum_{n=0}^{\infty}(n+1)!\left\Vert \tilde{f}_{n}\right\Vert _{\mathcal{L}_{T}^{2}((\mathrm{d}t\times\nu)^{n})}^{2}<\infty
\]
and has \emph{Skorohod integral with respect to $\widetilde{N}$
\[
\delta(X)=\int_{0}^{T}\int_{\mathbb{R}_{0}}X_{t,z}\widetilde{N}(\delta t,\mathrm{d}z):=\sum_{n=0}^{\infty}I_{n+1}(\tilde{f}_{n}).
\]
}
\end{defn*}

\begin{defn*}
Let $\mathbb{D}_{1,2}$ be the \emph{stochastic Sobolev space} consisting
of all $\mathcal{F}_{T}$-measureble random variables $F\in\mathcal{L}^{2}(\mathbb{P})$
with chaos expansion $F=\sum_{n=0}^{\infty}I_{n}(f_{n})$ satisfying
\[
\left\Vert F\right\Vert _{\mathbb{D}_{1,2}}:=\sum_{n=1}^{\infty}(n)n!\left\Vert \tilde{f}_{n}\right\Vert _{\mathcal{L}_{T}^{2}((\mathrm{d}t\times\nu)^{n})}^{2}<\infty.
\]
For every $F\in\mathbb{D}_{1,2}$ its \emph{Malliavin derivative }is
defined as
\[
D_{t,z}F:=\sum_{n=0}^{\infty}nI_{n-1}(f_{n}(\cdot,t,z)).
\]

\end{defn*}
Let us mention here that $Dom(\delta)\subseteq\mathcal{L}^{2}(\mathbb{P}\times\mathrm{d}t\times\nu)$,
$\delta(X)\in\mathcal{L}^{2}(\mathbb{P})$, $\mathbb{D}_{1,2}\subset\mathcal{L}^{2}(\mathbb{P})$
and $DF\in\mathcal{L}^{2}(\mathbb{P}\times\mathrm{d}t\times\nu)$. 

We have the following theorems are central for the results below;
for their proof and more details we refer to \cite{dNOkPro} and references
therein.
\begin{thm*}
(\emph{Duality formula}) Let $X$ be Skorohod integrable and let $F\in\mathbb{D}_{1,2}$.
Then
\[
\mathbb{E}\left[F\int_{0}^{T}\int_{\mathbb{R}_{0}}X_{t,z}\widetilde{N}(\delta t,\mathrm{d}z)\right]=\mathbb{E}\left[\int_{0}^{T}\int_{\mathbb{R}_{0}}X_{t,z}D_{t,z}F\mathrm{d}t\nu(\mathrm{d}z)\right].
\]

\end{thm*}

\begin{thm*}
(Product rule) Let $F,G\in\mathbb{D}_{1,2}$ with $G$ bounded. Then
$FG\in\mathbb{D}_{1,2}$ and 
\[
D_{s,z}(FG)=FD_{s,z}G+GD_{s,z}F+D_{s,z}FD_{s,z}G,\qquad\mathrm{d}t\times\nu-a.e.
\]

\end{thm*}

\begin{thm*}
(Chain rule) Let $F\in\mathbb{D}_{1,2}$, and let $g$ be a continuous
function such that $g(F)\in\mathcal{L}^{2}(\mathbb{P})$ and $g(F+D_{s,z}F)\in\mathcal{L}^{2}(\mathbb{P}\times\mathrm{d}t\times\nu)$.
Then $g(F)\in\mathbb{D}_{1,2}$ and 
\[
D_{s,z}g(F)=g\left(F+D_{s,z}F\right)-g(F).
\]

\end{thm*}

\subsubsection{A recursive formula}
\begin{lem}
\label{lem: the Malliavin lemma}For every deterministic Skorohod
integrable function $f$ and non-negative integer $n$, define
\[
F:=\int_{0}^{T}\int_{\mathbb{R}_{0}}f(s,z)\widetilde{N}(\mathrm{d}s,\mathrm{d}z),\qquad\text{and}\qquad X_{n}:=F^{n}\mathrm{e}^{-F}.
\]
If $Y\in\mathbb{D}_{1,2}$ is bounded, then the Malliavin derivative
of $YX_{n}$ is given ($\mathrm{d}t\times\nu-$a.e.) by
\[
D_{s,z}(YX_{n})=\mathrm{e}^{-f(s,z)}\left(Y+D_{s,z}Y\right)\left(\sum_{k=0}^{n}\binom{n}{k}X_{n-k}f^{k}(s,z)\right)-YX_{n}.
\]
\end{lem}
\begin{proof}
By the product rule we have
\[
D_{s,z}(YX_{n})=\left(D_{s,z}(Y\mathrm{e}^{-F})\right)\left(F^{n}+D_{s,z}F^{n}\right)+Y\mathrm{e}^{-F}D_{s,z}F^{n},\qquad\mathrm{d}t\times\nu-a.e,
\]
and
\[
D_{s,z}(Y\mathrm{e}^{-F})=YD_{s,z}\mathrm{e}^{-F}+(D_{s,z}Y)\left(\mathrm{e}^{-F}+D_{s,z}\mathrm{e}^{-F}\right),\qquad\mathrm{d}t\times\nu-a.e.
\]
Moreover, since $D_{s,z}F=f(s,z)$, then an application of the chain
rule tells us that $D_{s,z}\mathrm{e}^{-F}=\mathrm{e}^{-F}(\mathrm{e}^{-f(s,z)}-1)$
and 
\[
D_{s,z}F{}^{n}=(F+D_{s,z}F)^{n}-F^{n}=(F+f(s,z))^{n}-F^{n}=\sum_{k=0}^{n-1}\binom{n}{k}F^{k}f^{n-k}(s,z).
\]
Combining these expressions we get
\begin{align*}
D_{s,z}YX_{n} & =\left(Y\mathrm{e}^{-F}(\mathrm{e}^{-f(s,z)}-1)+(D_{s,z}Y)\left(\mathrm{e}^{-F}+\mathrm{e}^{-F}(\mathrm{e}^{-f(s,z)}-1)\right)\right)\sum_{k=0}^{n}\binom{n}{k}F^{k}f^{n-k}(s,z)\\
 & \quad+Y\mathrm{e}^{-F}\sum_{k=0}^{n-1}\binom{n}{k}F^{k}f^{n-k}(s,z)\\
 & =Y\mathrm{e}^{-F}\left((\mathrm{e}^{-f(s,z)}-1)\sum_{k=0}^{n}\binom{n}{k}F^{k}f^{n-k}(s,z)+\sum_{k=0}^{n-1}\binom{n}{k}F^{k}f^{n-k}(s,z)\right)\\
 & \quad+(D_{s,z}Y)\mathrm{e}^{-(F+f(s,z))}\sum_{k=0}^{n}\binom{n}{k}F^{k}f^{n-k}(s,z)\\
 & =Y\mathrm{e}^{-F}\left(\mathrm{e}^{-f(s,z)}\sum_{k=0}^{n}\binom{n}{k}F^{k}f^{n-k}(s,z)^{n-k}-F^{n}\right)\\
 & \quad+(D_{s,z}Y)\mathrm{e}^{-(F+f(s,z))}\sum_{k=0}^{n}\binom{n}{k}F^{k}f^{n-k}(s,z).
\end{align*}
Thus, rewritting the last equivalence in terms of $X_{1},...,X_{n}$,
we get the result.\end{proof}
\begin{lem}
\label{lem: The-conditional-characteristic}The conditional characteristic
function of such Hazard processes in (\ref{eq:abscont-1}),
\[
\varphi(u;t,T):=\mathbb{E}\left[\left.\mathrm{e}^{\mathrm{\mathrm{i}}u\Lambda_{T}}\right|\mathcal{F}_{t}\right],
\]
is given by
\begin{equation}
\varphi(u;t,T)=\exp\left\{ \mathrm{\mathrm{i}}u\Lambda_{t}+\int_{t}^{T}\int_{\mathbb{R}_{0}}\left[\mathrm{e}^{\mathrm{\mathrm{i}}u\sigma(s,z)}-1-\mathrm{\mathrm{i}}u\sigma(s,z)\right]\mathrm{d}s\nu(\mathrm{d}z)\right\} .\label{eq: characteristic exponent}
\end{equation}
\end{lem}
\begin{proof}
Since the integrands $\mu$ and $\sigma$ are deterministic, then
the increment $\Lambda_{T}-\Lambda_{t}$ is independent of $\mathcal{F}_{t}$
and
\[
\mathbb{E}\left[\left.\mathrm{e}^{\mathrm{\mathrm{i}}u\Lambda_{T}}\right|\mathcal{F}_{t}\right]=\exp\left\{ \mathrm{\mathrm{i}}u\Lambda_{t}\right\} \mathbb{E}\left[\mathrm{e}^{\mathrm{\mathrm{i}}u(\Lambda_{T}-\Lambda_{t})}\right].
\]
Notice that for every deterministic function $f$ the process $(\mathcal{E}_{t}(f))_{t\geq0}$
defined by 
\[
\mathcal{E}_{t}(f):=\exp\left\{ \int_{0}^{t}\int_{\mathbb{R}_{0}}f(s,z)\widetilde{N}(\mathrm{d}s,\mathrm{d}z)-\int_{0}^{t}\int_{\mathbb{R}_{0}}\left[\mathrm{e}^{f(s,z)}-1-f(s,z)\right]\mathrm{d}s\nu(\mathrm{d}z)\right\} 
\]
is a Dol\'eans-Dade exponental martingale. Thus $\mathbb{E}[\mathcal{E}_{T}(f)]=1$,
and so 
\[
\mathbb{E}\left[\mathcal{E}_{T}(f)\mathrm{e}^{\int_{0}^{T}\int_{\mathbb{R}_{0}}\left[\mathrm{e}^{f(s,z)}-1-f(s,z)\right]\mathrm{d}s\nu(\mathrm{d}z)}\right]=\mathrm{e}^{\int_{0}^{T}\int_{\mathbb{R}_{0}}\left[\mathrm{e}^{f(s,z)}-1-f(s,z)\right]\mathrm{d}s\nu(\mathrm{d}z)}.
\]
In our case this reads as 
\begin{align*}
\mathbb{E}\left[\mathrm{e}^{\mathrm{\mathrm{i}}u(\Lambda_{T}-\Lambda_{t})}\right] & =\mathbb{E}\left[\exp\left\{ \int_{0}^{T}\int_{\mathbb{R}_{0}}\mathrm{\mathrm{i}}u\mathbf{1}_{[t,T]}(s)\sigma(s,z)\widetilde{N}(\mathrm{d}s,\mathrm{d}z)\right\} \right]\\
 & =\exp\left\{ \int_{0}^{T}\int_{\mathbb{R}_{0}}\left[\mathrm{e}^{\mathrm{\mathrm{i}}u\mathbf{1}_{[t,T]}(s)\sigma(s,z)}-1-\mathrm{\mathrm{i}}u\mathbf{1}_{[t,T]}(s)\sigma(s,z)\right]\mathrm{d}s\nu(\mathrm{d}z)\right\} .
\end{align*}
It remains to notice that $\mathrm{e}^{\mathrm{\mathrm{i}}u\mathbf{1}_{[t,T]}\sigma}-1-\mathrm{\mathrm{i}}u\mathbf{1}_{[t,T]}\sigma=\left[\mathrm{e}^{\mathrm{\mathrm{i}}u\sigma}-1-\mathrm{\mathrm{i}}u\sigma\right]\mathbf{1}_{[t,T]}$.\end{proof}
\begin{lem}
\label{lem: recursive lemma}Under the notation of Lemma \ref{lem: the Malliavin lemma}
we have 
\[
\mathbb{E}\left[X_{0}\right]=\exp\left\{ \int_{0}^{T}\int_{\mathbb{R}_{0}}\left[\mathrm{e}^{-f(s,z)}-1+f(s,z)\right]\mathrm{d}s\nu(\mathrm{d}z)\right\} ,
\]
and for $n\geq1$ the following recursive formula holds true 
\begin{align*}
\mathbb{E}\left[X_{n+1}\right] & =\mathbb{E}[X_{n}]\int_{0}^{T}\int_{\mathbb{R}_{0}}(\mathrm{e}^{-f(s,z)}-1)f(s,z)\mathrm{d}s\nu(\mathrm{d}z)\\
 & \quad+\sum_{k=1}^{n}\binom{n}{k}\mathbb{E}[X_{n-k}]\int_{0}^{T}\int_{\mathbb{R}_{0}}\mathrm{e}^{-f(s,z)}f^{k+1}(s,z)\mathrm{d}s\nu(\mathrm{d}z).
\end{align*}
\end{lem}
\begin{proof}
The Lemma \ref{cor: char func} provides the the base case ($n=0$).
For $n\geq1$, notice that
\begin{align*}
\mathbb{E}[X_{n+1}] & =\mathbb{E}[FX_{n}]\\
 & =\mathbb{E}\left[\int_{0}^{T}\int_{\mathbb{R}_{0}}f(s,z)D_{s,z}X_{n}\mathrm{d}s\nu(\mathrm{d}z)\right]\\
 & =\mathbb{E}\left[\int_{0}^{T}\int_{\mathbb{R}_{0}}f(s,z)\left(\mathrm{e}^{-f(s,z)}\sum_{k=0}^{n}\binom{n}{k}X_{n-k}f^{k}(s,z)-X_{n}\right)\mathrm{d}s\nu(\mathrm{d}z)\right],
\end{align*}
where the second line follows from the duality formula, and the last
one from Lemma \ref{lem: the Malliavin lemma} by setting $Y=1$.
The result then follows by the linearity of the expectation.
\end{proof}

\subsubsection{Proof of Theorem \ref{thm: the recursive formula}}

Notice that (\ref{eq: n-th jump survival probability}) can be rewritten
as 
\begin{align*}
\mathbb{P}(\tau_{n}>T|\mathcal{F}_{t})= & \sum_{k=0}^{n-1}\frac{1}{k!}\mathbb{E}\left[\left.\Lambda_{T}^{k}\mathrm{e}^{-\Lambda_{T}}\right|\mathcal{F}_{t}\right]\\
= & \mathbf{1}_{\{\tau_{n}>t\}}\mathrm{e}^{\Lambda_{t}}\left(\sum_{k=0}^{n-1}\sum_{j=0}^{k}\frac{\Lambda_{t}^{j}}{j!(k-j)!}\mathbb{E}\left[\left.(\Lambda_{T}-\Lambda_{t})^{k-j}\mathrm{e}^{-(\Lambda_{T}-\Lambda_{t})}\right|\mathcal{F}_{t}\right]\right)
\end{align*}
Indeed, it suffices to expand the factor
\[
\Lambda_{T}^{k}=([\Lambda_{T}-\Lambda_{t}]+\Lambda_{t})^{k}=\sum_{j=0}^{k}\binom{k}{j}(\Lambda_{T}-\Lambda_{t})^{k-j}\Lambda_{t}^{j},
\]
and use that $\Lambda_{t}$ is $\mathcal{F}_{t}$-measurable. Now,
since the integrand in (\ref{eq: Levy driven Hazard process-1}) is
deterministic, we have that the increment $\Lambda_{T}-\Lambda_{t}$
is independent of $\mathcal{F}_{t}$ and thus 
\[
\mathbb{E}\left[\left.\left(\Lambda_{T}-\Lambda_{t}\right)^{k-j}\mathrm{e}^{-(\Lambda_{T}-\Lambda_{t})}\right|\mathcal{F}_{t}\right]=\mathbb{E}\left[\left(\Lambda_{T}-\Lambda_{t}\right)^{k-j}\mathrm{e}^{-(\Lambda_{T}-\Lambda_{t})}\right]=m_{k-j}.
\]
Applying Lemma \ref{lem: recursive lemma} with $f(s,z):=\mathbf{1}_{[t,T]}(s)\sigma(s,z)$
we show that the quantities $m_{0},m_{1},...,m_{n}$ satisfy the recursion
claimed. In order to remove factor $\mathbf{1}_{[t,T]}(s)$ from the
expression, it remains to take into account the basic identities
\[
\int_{0}^{T}\int_{\mathbb{R}_{0}}(\mathrm{e}^{-\mathbf{1}_{[t,T]}(s)\sigma(s,z)}-1)\mathbf{1}_{[t,T]}(s)\sigma(s,z)\mathrm{d}s\nu(\mathrm{d}z)=\int_{t}^{T}\int_{\mathbb{R}_{0}}(\mathrm{e}^{-\sigma(s,z)}-1)\sigma(s,z)\mathrm{d}s\nu(\mathrm{d}z),
\]
and
\[
\int_{0}^{T}\int_{\mathbb{R}_{0}}\mathrm{e}^{-\mathbf{1}_{[t,T]}(s)\sigma(s,z)}\mathbf{1}_{[t,T]}(s)\sigma^{n-k+1}(s,z)\mathrm{d}s\nu(\mathrm{d}z)=\int_{t}^{T}\int_{\mathbb{R}_{0}}\mathrm{e}^{-\sigma(s,z)}\sigma^{n-k+1}(s,z)\mathrm{d}s\nu(\mathrm{d}z).
\]

\end{document}